\documentclass[12pt, reqno]{amsart}
\usepackage[utf8]{inputenc}
\usepackage{graphicx}
\usepackage{amsmath}
\usepackage{amsthm}
\usepackage{amssymb,bbm}%
\usepackage[numbers, square]{natbib}
\usepackage{xcolor}

\theoremstyle{plain}
\newtheorem{theorem}{Theorem}[section]
\newtheorem{lemma}[theorem]{Lemma}
\newtheorem{corollary}[theorem]{Corollary}

\theoremstyle{definition}
\newtheorem{definition}[theorem]{Definition}

\theoremstyle{remark}

\newtheorem{conjecture}[theorem]{Conjecture}

\makeindex
\makeglossary

\newcommand{\E}{\textnormal{\textsf E}\,}
\newcommand{\R}{\mathbb{R}}

\newcommand{\Vol}{\mathop{\mathrm{Vol}}\nolimits}

\newcommand{\conv}{{\mathrm{conv}}}

\newcommand{\osc}{\mathop{\mathrm{osc}}\nolimits}

\begin{document}

\author[M.~Dospolova]{Maria Dospolova}
\address{Maria Dospolova, St.~Petersburg Department of Steklov Institute of Mathematics, Fontanka~27, 191011 St.~Petersburg, Russia}
\email{dospolova.maria@yandex.ru}

\author[M.~Germanskov]{Mikhail Germanskov}
\address{Mikhail Germanskov, St.~Petersburg Department of Steklov Institute of Mathematics, Fontanka~27, 191011 St.~Petersburg, Russia}
\email{mgermanskov@gmail.com}

\author[D.~Zaporozhets]{Dmitry Zaporozhets}
\address{Dmitry Zaporozhets, St.~Petersburg Department of Steklov Institute of Mathematics, Fontanka~27, 191011 St.~Petersburg, Russia}
\email{zap1979@gmail.com}

\title{On Steiner entire function}
\keywords{Steiner entire function, intrinsic volumes, Gaussian-bounded convex sets, entire functions of finite order, Gaussian continuity, Wiener spiral, Blaschke problem, Hilbert space geometry, Gaussian integrals, Hadamard factorization
}
\subjclass[2010]{52A05, 30D15, 60G15}
\thanks{This work was supported by the Ministry of Science and Higher Education of the Russian Federation (agreement 075-15-2025-344 dated 29/04/2025 for Saint Petersburg Leonhard Euler International Mathematical. The work of MD and DZ  was supported by the Theoretical Physics and Mathematics Advancement Foundation
``BASIS''}

\begin{abstract}
We introduce and study the Steiner entire function, an analytic generating function for the intrinsic volumes of a convex compact set in a Hilbert space. This function extends the classical Steiner polynomial to infinite dimensions and encodes key geometric information about the set. We establish fundamental results on its order, type, and canonical product representation, and show how its analytic growth properties characterize Gaussian continuity. In particular, we provide new criteria for this property in terms of entire function theory, disprove a conjecture of Gao and Vitale, and conjecture a new characterization of admissible intrinsic volume sequences in infinite dimensions.
\end{abstract}

\maketitle

\section{Introduction}
The intrinsic volumes, $V_k(K)$, of a convex body $K \subset \R^n$ are fundamental geometric quantities that generalize the notions of volume, surface area, and mean width. They appear as coefficients in the classic Steiner formula, which states that the volume of the Euclidean neighborhood of $K$ of radius $\lambda\ge0$ is a polynomial in $\lambda$:
\begin{equation}\label{eq:Steiner}
\Vol_d(K + \lambda B^d) = \sum_{k=0}^d \kappa_{d-k}\,V_k(K)\,\lambda^{\,d-k}, \qquad \lambda\ge0,
\end{equation}
where $B^d$ is the unit $d$-ball, $\kappa_j=\Vol_j(B^j)=\pi^{j/2}/\Gamma(\frac{j}{2}+1)$, and $V_0(K),\dots, V_d(K)$ are the (normalized) \emph{intrinsic volumes} of $K$. The values $V_k(K)$ serve as key size parameters of $K$: for example, $V_d(K)$ is the volume of $K$, $V_{d-1}(K)$ is half the surface area, $V_1(K)$ is, up to a constant, the mean width, and $V_0(K) = 1$. By definition, $V_k(K) = 0$ for $k > \dim K$. The normalization in \eqref{eq:Steiner} is chosen so that $V_k(K)$ are intrinsic to $K$, i.e., independent of the ambient dimension.

Thanks to this invariance, these notions have been extended to infinite-dimensional convex sets by Sudakov \cite{Sud76} and Chevet \cite{Chev76}. Let $H$ be a separable (real) Hilbert space.  For a convex compact $K$, the numbers $V_k(K)$ can be defined by an approximation procedure: if $K'$ ranges over all finite-dimensional compact convex subsets of $K$, one sets 
\[ V_k(K) \;:=\; \sup_{K'\subset K} V_k(K') \;\in [0,\infty]. \] 
This agrees with the finite-dimensional definition when $K$ itself is finite-dimensional. 

Note that for an infinite-dimensional convex compact $K$, the intrinsic volumes might be infinite. In fact, if $V_1(K)<\infty$ then automatically $V_k(K)<\infty$ for all $k$ (as shown by Sudakov~\cite[Theorem~1]{Sud76}), whereas if $V_1(K)=\infty$ then $V_k(K)=\infty$ for all $k\ge1$. 

We call a convex compact set $K$ \emph{Gaussian bounded (GB)} if $V_1(K)<\infty$, and hence $V_k(K)\in[0,\infty)$ for all $k\ge0$. This terminology comes from the equivalent probabilistic condition: there exists a version of an \emph{isonormal Gaussian process} on $K$—that is, a centered Gaussian family \( \{\xi(t) : t \in K\} \) with covariance structure given by
\begin{align}\label{eq:iso}
\mathbb{E}[\xi(t)\xi(t')] = \langle t,t' \rangle,
\end{align}
whose sample paths are almost surely bounded (see Section~\ref{sec:Gaussian-proc}). 
In particular, any convex body in $\R^d$ is GB, and many infinite-dimensional $K$ of practical interest are GB. Throughout the paper, when we refer to a set as \emph{Gaussian bounded (GB)}, we implicitly assume that it is both convex and compact.

If $K$ is GB, one can define its \emph{Wills functional} as the sum of all intrinsic volumes: 
\[ W(K) := \sum_{k=0}^\infty V_k(K). \] 
This sum is finite; indeed, Chevet \cite{Chev76} and McMullen~\cite{McM91} proved the inequality 
\begin{equation}\label{eq:Chevet}
V_k(K) \le \frac{(V_1(K))^k}{k!}, \qquad k=0,1,2,\dots,
\end{equation}
for any GB-set. It follows that 
\begin{equation}\label{eq:W-expV1}
W(K) \le \exp(V_1(K)) < \infty,
\end{equation}
so $W(K)$ is always finite for GB-sets. In fact, $W(K)$ admits a striking representation as a Gaussian integral: for any compact convex $K\subset \R^d$,
\begin{equation}\label{eq:Wills-Gauss}
W(K) = \int_{\R^d} \frac{1}{(2\pi)^{d/2}}\exp\!\Big(-\frac{\mathrm{dist}^2(x,K)}{2}\Big)dx,
\end{equation}
where $\mathrm{dist}(x,K)=\inf_{y\in K}\|x-y\|$. Equation~\eqref{eq:Wills-Gauss} (see \cite{Vit01,Bog98} for derivations) shows that $W(K)$ is the Gaussian volume of a \emph{tube} around $K$. Tsirelson~\cite[Theorem~5]{Tsi85} generalized this to infinite dimensions, proving an \emph{infinite-dimensional Steiner formula}: if $K\subset H$ is GB, then for $\lambda\geq0$,
\begin{equation}\label{eq:Tsirelson}
   \E\exp \left({\sup_{t \in K} \left[\sqrt{2\pi}\lambda \xi(t) - \pi\lambda^{2} \|t\|^2
      \right]}\right)  = \sum_{k=0}^{\infty} \lambda^k V_{k}(K),\quad \lambda\geq0,
\end{equation}
where $\xi$ is isonormal on $H$, see~\eqref{eq:iso}. In particular, setting $\lambda=1$ in \eqref{eq:Tsirelson} recovers $W(K)$. Equation~\eqref{eq:Tsirelson} serves as a generating function for the intrinsic volumes of $K$.

\subsection*{The Steiner entire function}
For a convex body $K\subset\R^d$, Cifre and  Nicol\'as~\cite{CifNic13} proposed studying the polynomial 
\[ f_K(z) = \sum_{k=0}^d V_k(K)\,z^k, \qquad z\in\mathbb{C}. \] 
They named $f_K(z)$ the \emph{Wills polynomial} of $K$, since $f_K(1)=W(K)$, and investigated its roots. By analogy, for an infinite-dimensional  GB-set $K\subset H$, we introduce:

\begin{definition}
The \emph{Steiner entire function} of $K$ is 
\[ f_K(z) := \sum_{k=0}^\infty V_k(K)\, z^k, \qquad z\in\mathbb{C}. \] 
\end{definition}
The name is motivated by the fact that, for real positive $z$, this function coincides with the right-hand side of the infinite-dimensional Steiner formula~\eqref{eq:Tsirelson}.

This power series has an infinite radius of convergence, as guaranteed by the bound \eqref{eq:Chevet}. Indeed, \eqref{eq:Chevet} implies 
$$
|f_K(z)| \le \sum_{k\ge0} \frac{(V_1(K)|z|)^k}{k!} = \exp(V_1(K)|z|)
$$ 
for all $z\in\mathbb{C}$. In particular, $f_K$ is an \emph{entire function}, i.e., analytic on $\mathbb{C}$). 

The first aim of this work is to investigate fundamental properties of the Steiner entire function $f_K(z)$ and to present concrete examples. The second aim is to connect analytic properties of $f_K$ to geometric properties of $K$. In particular,  we characterize the \emph{Gaussian continuity} of $K$ (a crucial regularity property in convex geometry and Gaussian process theory, see Subsection~\ref{sec:Gaussian-proc}) in terms of the growth of $f_K$: roughly speaking, $K$ is Gaussian continuous if and only if $f_K$ does \emph{not} have full exponential order. This provides an analytic criterion for a geometric and probabilistic phenomenon.

\subsection*{Organization of the paper}
In the next section, we collect the necessary background. Section~3 contains all main results. Subsection~4 presents examples and constructions, and the proofs are located in Section~5.

\section{Preliminaries}\label{sec:prelim}
\subsection{Ultra-log-concave property}

The sequence of intrinsic volumes of a GB-set $K$ is known to be ultra-log-concave, a result established by Favard in finite dimensions and extended to infinite dimensions by Chevet \cite{Chev76} and McMullen~\cite{McM91}. This property, which follows from the deep Alexandrov--Fenchel inequality (see, e.g.,~\cite[Section~4.2]{BZ80} or~\cite[Section~14.3]{Sch14}), is expressed as:
\begin{equation} \label{eq:ultra-log-concave}
V_k(K)^2 \ge \frac{k+1}{k} V_{k+1}(K) V_{k-1}(K), \quad \text{for } k \ge 1.
\end{equation}
This inequality implies \eqref{eq:Chevet} as a simple corollary.

\subsection{Gaussian bounded and Gaussian continuous convex sets}\label{sec:Gaussian-proc}
Let $H$ be a real Hilbert space with inner product $\langle\cdot,\cdot\rangle$ and norm $\|t\|=\sqrt{\langle t,t\rangle}$. We denote by $(\xi(t))_{t\in H}$ an \emph{isonormal Gaussian process} on $H$, that is, a centered Gaussian family with covariance $\mathbb{E}[\xi(t)\xi(t')] = \langle t,t'\rangle$. 
If $e_1,e_2,\ldots$ is an orthonormal basis of $H$ and $t = t_1e_1+t_2e_2+\ldots$, then  $\xi$ can be represented as 
\begin{align*}
\xi(t_1e_1+t_2e_2+\ldots) =t_1N_1+t_2N_2+\ldots,
\end{align*}
where $N_1,N_2,\ldots$ are independent standard Gaussian variables. The series converges almost surely, since $t_1^2+t_2^2+\ldots<\infty$. In $\R^d$, it could be written in the form $\xi(t) = \langle t,N\rangle$, where $N$ is standard Gaussian vector in $\R^d$.

We say $K$ is \emph{Gaussian bounded (GB)} if $\xi$ has a version which is bounded uniformly on $K$ with probability one. This is equivalent to the earlier definition   via $V_1(K)<\infty$, by Sudakov's result~\cite{Sud76}. 

A stronger condition is that the sample paths of $\xi$ can be chosen continuous on $K$. Formally, $K$ is \emph{Gaussian continuous (GC)} if there exists a modification of the process $\{\xi(t): t\in K\}$ which is (almost surely) a continuous function on $K$. 

For GB-sets, Vitale characterized the subclass of GC-sets via the asymptotic behavior of the sequence of ratios of intrinsic volumes. Let
\begin{equation} \label{eq:mk_def}
m_k(K) := \frac{(k+1)V_{k+1}(K)}{V_k(K)}, \quad k \ge 0.
\end{equation}
\noindent
By convention, we set $\frac{0}{0} := 0$, so for finite-dimensional sets $K$, the sequence $m_k(K)$ eventually terminates (i.e., becomes zero) when $V_k(K) = 0$ for $k > \dim K$.

The sequence $\{m_k(K)\}_{k \ge 0}$ is non-increasing due to \eqref{eq:ultra-log-concave}, and thus has a limit (positive or zero).

Vitale showed that, for  a GB-set,
\begin{align}\label{thm:Vitale}
K\text{ is a GC-set if and only if} \lim_{k\to\infty} m_k(K) = 0.
\end{align}
In fact, he did more, connecting this limit with the \emph{oscillation} of $K$.

\subsection{Oscillation}
Let $\xi(t)$ be an isonormal process on a GB-set $K$. It\^o and Nisio~\cite{IN68} showed that there exists a deterministic function $\alpha : K \to [0, \infty]$ such that, with probability one, for all $t \in K$,
\[
\alpha(t) = \lim_{\varepsilon \to 0} \sup \{ |\xi(s) - \xi(s')| : s,s' \in B(t, \varepsilon) \} < \infty,
\]
where $B(t, \varepsilon)$ denotes the open ball of radius $\varepsilon$ centered at $t$.

The process $\xi(t)$ is continuous at a point $t \in K$ if and only if $\alpha(t) = 0$.

We define the \emph{oscillation} of the set 
$K$ as
\[
\osc(K) = \sup_{t \in K} \alpha(t)
\]
Clearly, $\osc(K) = 0$ if and only if $K$ is a $GC$-set. 

Vitale showed that the oscillation of $K$ is determined by its intrinsic volumes:
\begin{align*}
\osc(K) = \lim_{k\to\infty} m_k(K).
\end{align*}
In particular, we recover the GC characterization~\eqref{thm:Vitale} as the special case when $\osc(K)$ = 0.

\subsection{Entire function basics}\label{sec:entire-func}
We recall standard notions from the theory of entire functions (see \cite{Boa54,Lev96} for background). 

A complex-valued function $f(z)$ is said to be \emph{entire} if it is analytic throughout the entire complex plane. Every entire function admits  a globally convergent power series expansion of the form
\begin{equation}
f(z) = \sum_{n=0}^\infty a_n z^n,
\label{eq:entire_series}
\end{equation}
with an infinite radius of convergence.

According to the Cauchy–Hadamard theorem, this implies that the sequence of coefficients satisfies
\begin{equation*}
\lim_{n \to \infty} \sqrt[n]{|a_n|} = 0.
\label{eq:cauchy_hadamard}
\end{equation*}

The \textit{order} $\rho$ of an entire function $f(z)$ is defined as the infimum of all real numbers $\mu$ for which  
\[
f(z) = O\left(e^{|z|^\mu}\right) \quad \text{as } |z| \to \infty.
\]
The order can also be computed via the formula
\[
\rho = \limsup_{r \to \infty} \frac{\ln\ln M_f(r) }{\log r},
\]
where
\[
M_f(r) = \max_{|z| = r} |f(z)|.
\]

A more refined measure of the growth of an entire function of order $\rho$ is given by its \textit{type}.  
The {type} $\sigma$ of an entire function $f(z)$ of order $\rho$ is defined as the infimum of all positive numbers $A$ such that
\begin{align} \label{type}
M_f(r) < e^{A r^{\rho}} \quad \text{as } r \to \infty.
\end{align}
Equivalently, the type $\sigma$ can be expressed as
\[
\sigma = \limsup_{r \to \infty} \frac{\ln M_f(r)}{r^\rho}.
\]

The order $\rho$ of an entire function $f(z)$ given by the power series~\eqref{eq:entire_series} is determined by the formula
\begin{equation}
\rho = \lim_{n \to \infty} \frac{n \ln n}{\ln \left( \frac{1}{|a_n|} \right)}.
\label{eq:order_formula}
\end{equation}

If the function $f(z)$ has finite positive order $0 < \rho < \infty$, then its \emph{type} $\sigma$ is given by
\begin{equation}
(\sigma e \rho)^{1/\rho} = \lim_{n \to \infty} n^{1/\rho} \sqrt[n]{|a_n|}.
\label{eq:type_formula}
\end{equation}

Entire functions of finite order admit a representation in the form of an infinite product. This result is known as Hadamard’s factorization theorem: 
An entire function $f(z)$ of finite order $\rho$ can be written as
\begin{align}\label{eq:Hadamard}
f(z) = z^m e^{P(z)} \prod_{k=1}^\infty \left(1 - \frac{z}{z_k}\right) \exp\left( \frac{z}{z_k} + \cdots + \frac{1}{p} \left(\frac{z}{z_k}\right)^p \right),
\end{align}
where $z_k$ are the nonzero zeros of $f(z)$, $m$ is the multiplicity of the zero at the origin, $P(z)$ is a polynomial of degree $q \le \rho$, and $p$ is the smallest non-negative integer such that the series
\[
\sum_{k=1}^\infty \frac{1}{|z_k|^{p+1}}
\]
converges. The part 
\begin{align*}
g(z) =  \prod_{k=1}^\infty \left(1 - \frac{z}{z_k}\right) \exp\left( \frac{z}{z_k} + \cdots + \frac{1}{p} \left(\frac{z}{z_k}\right)^p \right),
\end{align*}
is called the \emph{canonical product}. The infimum over all $\lambda > 0$ for which the series
\[
\sum_{k=1}^\infty \frac{1}{|z_k|^{\lambda}}
\]
converges is called the \emph{convergence exponent} of the sequence $\{z_k\}$. This quantity can also be expressed as
\[
\lambda = \limsup_{n \to \infty} \frac{\ln n}{\ln |z_n|},
\]
see, e.g.,~\cite[Theorem~3.4]{But14}.

The classical Borel--Levin theorem~\cite[Section~4.3, Theorem~3]{Lev96} establishes that, for canonical products, the convergence exponent coincides with the order of the corresponding entire function. In particular, for the canonical product $g(z)$, we have
\begin{align}\label{eq:orderlimsup}
\rho = \limsup_{n \to \infty} \frac{\ln n}{\ln |z_n|}.
\end{align}


\section{Main results}
\subsection{Growth and zeros of the Steiner entire function}\label{subsec:growth}
Our first result describes the fundamental growth indicators (order and type) of $f_K(z)$. It shows that $f_K$ behaves at most like an exponential, never faster. This is an infinite-dimensional analog of the finite-degree property of Steiner's polynomial \eqref{eq:Steiner}.

\begin{theorem}\label{thm:order-type}
Let $K\subset H$ be a  GB-set. Then the Steiner entire function $f_K(z)$ has order $\rho(K)\le 1$. If $\rho(K)=1$, then the type satisfies $\sigma(K)\le V_1(K)$. Moreover, the full range of orders $\rho \in [0,1]$ is realizable by suitable choices of $K$.
\end{theorem}

From Theorem~\ref{thm:order-type} and the Hadamard factorization in Section~\ref{sec:entire-func}, we immediately get:

\begin{corollary}\label{cor:Hadamard}
For any convex GB-body $K\subset H$, the Steiner entire function $f_K(z)$ can be represented in the form 
\begin{equation}\label{eq:fK-factor}
f_K(z) = e^{cz} \prod_{j=1}^\infty \Big(1 - \frac{z}{z_j}\Big)\exp\!\Big(\frac{z}{z_j}\Big),
\end{equation}
where $c\in \mathbb{C}$ and $\{z_j\}_{j\ge1}$ are the nonzero zeros of $f_K$. If $\sum_{j=1}^\infty \frac{1}{|z_j|}$ converges, then \eqref{eq:fK-factor} simplifies to 
\begin{equation}\label{eq:fK-simple-prod}
f_K(z) = e^{cz} \prod_{j=1}^\infty \Big(1 - \frac{z}{z_j}\Big).
\end{equation}
\end{corollary}

\begin{proof}
This is a direct consequence of Theorem~\ref{thm:order-type} and the discussion in Section~\ref{sec:entire-func} (Hadamard's theorem). The fact that $f_K(0)=1$ ensures no constant factor in front of \eqref{eq:fK-factor}. We emphasize that $c=0$ whenever $\rho(K)<1$ (since then $f_K$ has minimal type $0$ and genus 0). 
\end{proof}

\subsection{Steiner function and geometric properties}

The next result is a new set of criteria for Gaussian continuity and a formula for oscillation, expressed in terms of the analytic properties of the Steiner function.

\begin{theorem}\label{thm:gc_criterion}
Let $K$ be a GB-set in a separable Hilbert space $H$, and let $f_K(z)$ be its Steiner entire function with order $\rho(K)$ and type $\sigma(K)$.
\begin{enumerate}
    \item $K$ is a GC-set if and only if either $\rho(K) < 1$, or $\rho(K)=1$ and $\sigma(K)=0$.
    \item If $\rho(K)=1$, then the type of $f_K(z)$ is equal to the oscillation of $K$:
    $$
    \sigma(K) = \osc(K).
    $$
\end{enumerate}
\end{theorem}

These results provide an analytic criterion for Gaussian continuity: one can determine whether $K$ is GC by examining the growth rate of $f_K$ in the complex plane. 

Our next theorem shows that the decay rate of $m_k$ is directly related to the order of $f_K$. In particular, for the Wiener spiral examples with $\rho = 2/3$, we recover $m_k = O(k^{-1/2})$.

\begin{theorem} \label{thm:mk_asymptotics}
Let $K$ be a GB-set with Steiner function of order $\rho(K)\in[0,1]$. Then
$$\limsup_{k \to \infty} \frac{\ln m_k(K)}{\ln k} = 1-\frac{1}{\rho(K)} .$$
By convention, if $\rho(K) = 0$, the right-hand side is interpreted as $-\infty$.
\end{theorem}

Using this theorem, we will disprove the \emph{Gao--Vitale conjecture}.

\subsection{Gao--Vitale conjecture}
Gao and Vitale~\cite{GV01} initiated the study of the asymptotic behavior of $V_k(K)$ and $m_k(K)$ for GB-sets $K$ in a Hilbert space. They computed these quantities explicitly for the closed convex hull of the Wiener spiral (see Subsection~4.1), and observed a specific decay rate $m_k = O(k^{-1/2})$ in that case. Based on this and related evidence, Gao and Vitale formulated the following conjecture:

\begin{conjecture}[Gao--Vitale (2001)~\cite{GV01}]
For every convex GB-compact set $K$ in a separable Hilbert space, either
\[
\lim_{k \to \infty} m_k(K) > 0,
\]
or
\[
m_k(K) = O(k^{-1/2}) \quad (k \to \infty).
\]
In other words, the sequence $m_k(K)$ cannot decay to zero at a rate slower than $k^{-1/2}$.
\end{conjecture}

However, Theorem~\ref{thm:mk_asymptotics} implies that for any $\varepsilon > 0$, the estimate
\[
m_k(K) = O\left(k^{1 - 1/{\rho(K)} - \varepsilon}\right)
\]
fails whenever $\rho(K) > \frac{2}{3}$. By Theorem~\ref{thm:order-type}, such GB-sets $K$ do exist. Therefore, any GB-set with $\rho(K) > \frac{2}{3}$ provides a counterexample to the Gao--Vitale conjecture.

\subsection{Conjecture on the Blaschke diagram in Hilbert space}

To conclude this section, we propose a conjecture characterizing the possible sequences of intrinsic volumes of infinite-dimensional convex compacts. In 1916, Blaschke posed a problem that, in terms of intrinsic volumes, can be reformulated as follows: for convex compacts in $\mathbb{R}^3$, describe completely the set of triples $(V_1(K),\,V_2(K),\,V_3(K))$ arising from convex compact sets $K$. 

A well-known system of inequalities relating $V_1(K)$, $V_2(K)$, and $V_3(K)$ is derived from the  Aleksandrov--Fenchel inequalities (see, e.g.,~\cite[Section~4.2]{BZ80} or~\cite[Section~14.3]{Sch14}). However, this system does not yield an exhaustive set of necessary and sufficient conditions, and Blaschke’s problem remains open for $d \geq 3$.

We suspect that the infinite-dimensional analogue admits a simpler and more elegant characterization. This belief is supported by the observation that, in contrast to the finite-dimensional setting, all known asymptotic (dimension-free) inequalities for intrinsic volumes follow from a single inequality, namely~\eqref{eq:ultra-log-concave}.

\begin{conjecture}[Characterization of $V_k$-sequences in infinite dimensions]\label{conj:sequence}
A sequence of positive numbers $a_0 = 1, a_1, a_2, \dots$ is the sequence of intrinsic volumes $\{V_k(K)\}_{k \ge 0}$ for some infinite-dimensional  GB-set if and only if it satisfies the inequality
\[
a_k^2 \,\ge\, \frac{k+1}{k}\, a_{k-1} a_{k+1}, \qquad k = 1, 2, \dots.
\]
In other words, the volume sequence must be ultra-log-concave (as in ~\eqref{eq:ultra-log-concave}), and we conjecture that this condition is also sufficient in infinite dimensions.
\end{conjecture}

\section{Examples}\label{sec:examples}
We compute Steiner entire functions in three cases that demonstrate different growth regimes $\rho$. The first two deal with convex hulls of certain infinite curves in $\R^2$ (related to Brownian motion), yielding intermediate order $0<\rho<1$. The third is an infinite box in $H=\ell^2$, which attains the boundary case $\rho=1$ but allows explicit formulae. 

\subsection{Closed convex hull of the Wiener spiral}
Let 
\[ S := \{ \mathbf{1}_{[0,t]}(\cdot) : t\in[0,1]\} \subset L^2[0,1], \] 
where $\mathbf{1}_{[0,t]}(x)$ is the indicator function of the interval $[0,t]$. This curve in $L^2[0,1]$ is often referred to as the \emph{Wiener spiral}, as the isonormal process restricted to $S$ has the same distribution as Brownian motion. This follows immediately from the observation that the inner product of points on the curve coincides with the covariance function of the Wiener process.

Consider $K = \overline{\conv}(S)$, the closed convex hull of this spiral in $L^2$. It is known that $K$ is a compact subset of $L^2$ (in fact, a \emph{GC}-compact, as Brownian motion is continuous). Gao and Vitale~\cite{GV01} computed the intrinsic volumes of $K$ explicitly, showing that 
\[
V_k(K) = \frac{\kappa_k}{k!},
\]
which implies
\[
m_k(K) = \frac{(k+1)V_{k+1}(K)}{V_k(K)} = \sqrt\pi\, k^{-1/2}(1+o(1)) \quad \text{as } k\to\infty.
\] 
It is straightforward to check that the Steiner entire function of the closed convex hull of the Wiener spiral has the following representation:
\[
f_K(z) = \sum_{k = 0}^{\infty} \frac{\pi^{k/2}}{\Gamma\left( \frac{k}{2} + 1 \right) k!} z^k 
= {}_0F_2\left(\tfrac{1}{2}, 1; \tfrac{\pi z^2}{4} \right) 
+ 2z \cdot {}_0F_2\left(\tfrac{3}{2}, \tfrac{3}{2}; \tfrac{\pi z^2}{4} \right),
\]
where 
\[
{}_0F_2\left(\alpha, \beta; z \right) = \sum_{k = 0}^{\infty} \frac{z^k}{(\alpha)_k (\beta)_k k!}
\]
is the generalized hypergeometric function and 
\[
(\alpha)_k = \frac{\Gamma(\alpha + k)}{\Gamma(\alpha)} = \alpha (\alpha + 1) \cdots (\alpha + k - 1)
\]
is the Pochhammer symbol. For a detailed treatment of hypergeometric functions, see~\cite{Bat53}.

Applying Stirling's approximation with~\eqref{eq:order_formula} and~\eqref{eq:type_formula} immediately gives that $f_K(z)$ has order 
\[
\rho(K) = \frac{2}{3}
\quad \text{and type} \quad 
\sigma(K) = \frac{3}{2} (2\pi)^{1/3}.
\]

\subsection{Closed convex hull of the Wiener spiral bridge}
Now let
\[
S^\circ = \{1_{[0,t]} - t : t \in [0,1]\} \subset L^2[0,1]
\]
be the so-called \emph{Wiener spiral bridge}—another curve in $L^2[0,1]$ closely related to the Wiener spiral. The isonormal process $\xi$, when restricted to $S^\circ$, has the same distribution as the Brownian bridge. Let $K^\circ = \overline{\conv}(S^\circ)$. Gao showed that
$$
V_k(K^\circ)=\frac{\kappa_{k+1}}{2k!},
$$
which, as in the case of the Wiener spiral, implies
\[ m_k(K^\circ) = \frac{(k+1)V_{k+1}(K^\circ)}{V_k(K^\circ)} = \sqrt\pi k^{-1/2}(1+o(1)) \quad \text{as } k\to\infty. \] 
Again, similarly to the Wiener spiral, we have
\begin{align*}
f_{K^\circ}(z) = {}_0F_2 \left(\frac{1}{2}, \frac{3}{2}; \frac{\pi z^2}{4}\right) + \frac{1}{2} \pi z \cdot {}_0F_2 \left( \frac{3}{2}, 2; \frac{\pi z^2}{4}\right), 
\end{align*}
and the function $f_{K^\circ}(z)$ has the same order and type: $$\rho(K^\circ) = \frac{2}{3},\quad \sigma(K^\circ) = \frac{3}{2} (2\pi)^{1/3}.$$

\subsection{Infinite-dimensional parallelepipeds} \label{subsec:parall}

Consider 
\[ K = \prod_{j=1}^\infty [0,\ell_j]\subset H=\ell^2, \] 
an infinite box  with side lengths $\ell_1,\ell_2,\ell_3,\dots$ such that $\sum_{j=1}^\infty \ell_j^2 < \infty$. This set $K$ is convex and compact. In addition, $K$ is a GB-set if and only if $\sum_{j=1}^\infty \ell_j < \infty$, by Sudakov’s theorem (since $V_1(K) = \sum_j \ell_j$, as we see now). 

For this infinite parallelepiped, the intrinsic volumes $V_k(K)$ are given by the sums of all $k$-fold products of distinct side lengths. That is,
\[ V_k(K) = \sum_{1\le j_1 < j_2 < \cdots < j_k} \ell_{j_1}\ell_{j_2}\cdots \ell_{j_k}\,. \] 
This can be justified  by an approximation argument taking finite-dimensional projections of $K$. The formula is identical to that for a $d$-dimensional rectangular box when $d=\infty$, as long as the series converge.

Observe that this expression coincides with the expansion of the infinite product
\[
\prod_{j=1}^\infty (1 + \ell_j z) \,=\, 1 + \left( \sum_j \ell_j \right) z + \sum_{i<j} \ell_i \ell_j z^2 + \cdots = \sum_{k=0}^\infty V_k(K)\, z^k,
\]
for those values of $z$ for which the product converges. Since $\sum_j \ell_j < \infty$, the product $\prod_j (1 + \ell_j z)$ converges absolutely for all $z \in \mathbb{C}$, by comparison with $\exp(|z| \sum_j \ell_j)$. This yields the following identity for the Steiner entire function:
\[
f_K(z) = \prod_{j=1}^\infty (1 + \ell_j z).
\]

Applying the general formula for the order~\eqref{eq:orderlimsup}, we obtain:
\[
\rho\bigg(\prod_{j=1}^\infty [0,\ell_j]\bigg) = \limsup_{j \to \infty} \frac{\ln j}{\ln (1/\ell_j)}.
\]

In particular, if $\ell_j$ decays as a power law $\ell_j = j^{-\alpha}$ with $\alpha > 1$ (ensuring convergence of $\sum_j \ell_j$), then
\[
\rho = \frac{1}{\alpha}.
\]
Moreover, the sequences $\ell_j = e^{-j}$ and $\ell_j = \frac{1}{j(\ln j)^2}$ yield orders $0$ and $1$, respectively. Thus, the full range $\rho \in [0,1]$ described in Theorem~\ref{thm:order-type} is attained.

\section{Proofs}

\begin{proof}[Proof of Theorem~\ref{thm:order-type}]

By inequality (3), we have
\[
|f_K(z)| \leq f_K(|z|) = \sum_{k=0}^\infty V_k(K)|z|^k \leq \sum_{k=0}^\infty \frac{V_1(K)^k}{k!}|z|^k = \exp(V_1(K)|z|).
\]
Thus, the order $\rho(K) \leq 1$.

If $\rho(K) = 1$, then from the definition (11) and the inequality above, it follows that the type $\sigma(K) \leq V_1(K)$.

The realizability of all orders $\rho \in [0,1]$ is shown by the parallelepiped examples in Subsection~\ref{subsec:parall}.

\end{proof}

In what follows, we will need the Stolz-Cesàro theorem.

\begin{theorem} \label{thm:stolz_cesaro}
Let $\{a_n\}_{n \ge 1}$ and $\{b_n\}_{n \ge 1}$ be two sequences of real numbers. If $\{b_n\}$ is strictly monotone and unbounded, then
$$
\liminf_{n \to \infty} \frac{a_{n+1}-a_n}{b_{n+1}-b_n} \le \liminf_{n \to \infty} \frac{a_n}{b_n} \le \limsup_{n \to \infty} \frac{a_n}{b_n} \le \limsup_{n \to \infty} \frac{a_{n+1}-a_n}{b_{n+1}-b_n}.
$$
In particular, if $\lim_{n \to \infty} \frac{a_{n+1}-a_n}{b_{n+1}-b_n}$ exists, then $\lim_{n \to \infty} \frac{a_n}{b_n}$ also exists and they are equal.
\end{theorem}

We begin with a key lemma relating the order of $f_K$ of the intrinsic volumes via $m_k$.

\begin{lemma}
\label{lem:order-vk}
If $K$ is GB, then
\[
\rho(K) = \limsup_{k \to \infty} \frac{\ln k}{-\ln(m_k/k)}.
\]
\end{lemma}

\begin{proof}

It follows from~\eqref{eq:order_formula} and~\eqref{eq:type_formula} that
\begin{equation}
\rho(K) = \lim_{n \to \infty} \frac{n \ln n}{\ln \left( \frac{1}{|V_n|} \right)}
\label{eq:K_order_formula}
\end{equation}
and 
\begin{equation}
(\sigma(K) e \rho(K))^{1/\rho} = \lim_{n \to \infty} n^{1/\rho} \sqrt[n]{|V_n|}.
\label{eq:K_type_formula}
\end{equation}

Using the log-convexity of $\{-\ln V_k\}$ and the fact that $V_0 = 1$, we have
\[
-\ln\frac{V_{k+1}}{V_k} = -\ln V_{k+1} + \ln V_k \geq \frac{-\ln V_k}{k},
\]
hence
\[
\frac{k \ln k}{-\ln V_k} \geq \frac{\ln k}{-\ln(V_{k+1}/V_k)}.
\]

For the reverse inequality, since $\{-\ln V_k\}$ is convex and unbounded (as $\rho > 0$), it is eventually monotone increasing. Applying the Stolz--Ces\`aro theorem to the sequences $\{k \ln k\}$ and $\{-\ln V_k\}$, we get the desired equality. The third expression follows from $m_k = (k+1)V_{k+1}/V_k$.
\end{proof}

\begin{proof}[Proof of Theorem~\ref{thm:gc_criterion}]
By Lemma~\ref{lem:order-vk}, if $\rho < 1$, then
\[
\limsup_{k \to \infty} \frac{\ln m_k}{\ln k} = \rho - 1 < 0.
\]
This implies $m_k \to 0$, so $K$ is GC by~\eqref{thm:Vitale}.

For order $\rho = 1$, the type is given by
\[
e\sigma = \limsup_{n \to \infty} n|V_n|^{1/n}.
\]
Using $V_n = \frac{1}{n!}\prod_{j=0}^{n-1} m_j$, we compute
\[
n|V_n|^{1/n} = \frac{n}{(n!)^{1/n}} \left(\prod_{j=0}^{n-1} m_j\right)^{1/n}.
\]
By Stirling's formula, $\lim_{n \to \infty} \frac{n}{(n!)^{1/n}} = e$. Also, since $\lim_{j \to \infty} m_j = \mathrm{osc}(K)$, by Ces\`aro's theorem,
\[
\left(\prod_{j=0}^{n-1} m_j\right)^{1/n} = \exp\left(\frac{\sum_{j=0}^{n-1} \ln m_j}{n}\right) \to \mathrm{osc}(K).
\]
Therefore $\sigma = \mathrm{osc}(K)$.
\end{proof}

\begin{proof}[Proof of Theorem~\ref{thm:mk_asymptotics}]
From Lemma~\ref{lem:order-vk}, we have
\[
\rho = \limsup_{k \to \infty} \frac{\ln k}{-\ln(m_k/k)}.
\]
Rearranging and using properties of $\limsup$ and $\liminf$:
\[
\frac{1}{\rho} = \liminf_{k \to \infty} \frac{-\ln(m_k/k)}{\ln k} = \liminf_{k \to \infty} \left(1 - \frac{\ln m_k}{\ln k}\right).
\]
Therefore
\[
\limsup_{k \to \infty} \frac{\ln m_k}{\ln k} = 1 - \frac{1}{\rho} = \frac{\rho - 1}{\rho}.
\]
\end{proof}


\begin{thebibliography}{99}

\bibitem{Bat53}
H.~Bateman, \emph{Higher Transcendental Functions, Vol.~I}, McGraw-Hill, New York, 1953.

\bibitem{Boa54}
R.~P.~Boas, Jr., \emph{Entire Functions}, Academic Press, New York, 1954.

\bibitem{Bog98}
V.~I.~Bogachev, \emph{Gaussian Measures}, Amer. Math. Soc., Providence, RI, 1998.

\bibitem{Bray24}
G.~G.~Braychev, \emph{On zeros and Taylor coefficients of an entire function of logarithmic growth}, Ufimsk. Mat. Zh. \textbf{16} (2024), no.~2, 16--26 (in Russian).

\bibitem{BZ80}
Yu.~D.~Burago and V.~A.~Zalgaller, \emph{Geometric Inequalities}, Nauka, Leningrad, 1980; English transl., Springer, 1988.

\bibitem{But14}
S.~A.~Buterin, G.~Freiling, and V.~A.~Yurko, \emph{Lectures on the Theory of Entire Functions}, Schriftenreihe der Fakultät für Mathematik, SM-UDE-779, Universität Duisburg-Essen, 2014.

\bibitem{Chev76}
S.~Chevet, \emph{Processus Gaussiens et volumes mixtes}, Z. Wahrscheinlichkeitstheorie verw. Gebiete \textbf{36} (1976), no.~1, 47--65.

\bibitem{CifNic13}
M.~A.~Hern\'andez Cifre and J.~Yepes Nicol\'as, \emph{On the roots of the Wills functional}, J. Math. Anal. Appl. \textbf{401} (2013), no.~2, 733--742.

\bibitem{Fav33}
J.~Favard, \emph{Sur les corps convexes}, J. Math. Pures Appl. (9) \textbf{12} (1933), 219--282.

\bibitem{GaoVit01}
F.~Gao and R.~A.~Vitale, \emph{Intrinsic volumes of the convex hull of a random walk}, Electron. Commun. Probab. \textbf{6} (2001), 73--82.

\bibitem{GV01}
F.~Gao and R.~A.~Vitale, \emph{Intrinsic volumes of the Brownian motion body}, Discrete Comput. Geom. \textbf{26} (2001), no.~1, 41--50.

\bibitem{IN68}
K.~Itô and M.~Nisio, \emph{On the oscillation functions of Gaussian processes}, Math. Scand. \textbf{22} (1968), 209--223.

\bibitem{KR97}
D.~A.~Klain and G.-C.~Rota, \emph{Introduction to Geometric Probability}, Cambridge Univ. Press, 1997.

\bibitem{Lev96}
B.~Ya.~Levin, \emph{Lectures on Entire Functions}, Transl. Math. Monogr., vol.~150, Amer. Math. Soc., Providence, RI, 1996.

\bibitem{McM91}
P.~McMullen, \emph{Inequalities between intrinsic volumes}, Monatsh. Math. \textbf{111} (1991), no.~1, 47--53.

\bibitem{Sch14}
R.~Schneider, \emph{Convex Bodies: The Brunn--Minkowski Theory}, 2nd ed., Encyclopedia of Math. Appl., vol.~151, Cambridge Univ. Press, 2014.

\bibitem{Sud76}
V.~N.~Sudakov, \emph{Geometric problems in the theory of infinite-dimensional probability distributions}, Proc. Steklov Inst. Math. \textbf{141} (1976), 1--178.

\bibitem{Tsi85}
B.~S.~Tsirelson, \emph{The density of the maximum of a Gaussian process}, Teor. Veroyatnost. i Primenen. \textbf{30} (1985), no.~4, 730--738.

\bibitem{Vit01a}
R.~A.~Vitale, \emph{Intrinsic volumes and Gaussian processes}, Adv. Appl. Probab. \textbf{33} (2001), no.~2, 354--364.

\bibitem{Vit01}
R.~A.~Vitale, \emph{The Wills functional and Gaussian processes}, Ann. Probab. \textbf{29} (2001), no.~1, 448--457.

\bibitem{Vit10}
R.~A.~Vitale, \emph{On the oscillation of a Gaussian process}, in: \emph{High Dimensional Probability V: The Luminy Volume}, Inst. Math. Stat. Collect., vol.~5, IMS, Beachwood, OH, 2009, pp.~219--221.

\bibitem{Wil73}
J.~M.~Wills, \emph{Zur Gitterpunktanzahl konvexer Mengen}, Elem. Math. \textbf{28} (1973), 57--63.

\end{thebibliography}
\end{document}